\documentclass[preprint,hidelinks,onefignum,onetabnum,reqno]{siamart250211}

\usepackage{lipsum}
\usepackage{amsfonts}
\usepackage{graphicx}
\usepackage{epstopdf}
\usepackage{algorithmic}
\usepackage{subfig}
\usepackage{cleveref}

\ifpdf
  \DeclareGraphicsExtensions{.eps,.pdf,.png,.jpg}
\else
  \DeclareGraphicsExtensions{.eps}
\fi

\newsiamremark{remark}{Remark}
\newsiamremark{hypothesis}{Hypothesis}
\crefname{hypothesis}{Hypothesis}{Hypotheses}
\newsiamthm{claim}{Claim}
\newsiamremark{fact}{Fact}
\crefname{fact}{Fact}{Facts}

\headers{Functional tensor train neural networks}{Y. Feng, M. K. Ng, K. Tang and Z. Zhang}

\title{Functional tensor train neural network for solving high-dimensional PDEs\thanks{Submitted to the editors DATE.
\funding{M. Ng’s research is supported in part by the National Key Research and Development Program of China under Grant 2024YFE0202900, GDSTC: Guangdong and Hong Kong Universities “1+1+1” Joint Research Collaboration Scheme UICR0800008-24, Hong Kong RGC grant project (No. 12300125), and Joint NSFC and RGC N-HKU769/21. Z. Zhang’s research is supported in part by the National Natural Science Foundation of China (No. 92470103), Hong Kong RGC grant project (No. 17300325), and Seed Funding for Strategic Interdisciplinary Research Scheme 2021/22 (HKU). Kejun Tang is supported by the Natural Science Foundation of China  (No. 12501598).}}}
\author{Yani Feng\thanks{Department of Mathematics, The University of Hong Kong
  (\email{fengyn@hku.hk}, \email{zhangzw@hku.hk}).}
\and Michael K. Ng\thanks{Department of Mathematics, Hong Kong Bapist University
  (\email{michael-ng@hkbu.edu.hk}).}
\and Kejun Tang\thanks{Department of Mathematics, School of Sciences, Great Bay University
  (\email{tangkj@gbu.edu.cn}).}
\and Zhiwen Zhang\footnotemark[2]}

\usepackage{amsopn}

\ifpdf
\hypersetup{
  pdftitle={Functional tensor train neural network for solving high-dimensional PDEs},
  pdfauthor={Y. Feng, M. K. Ng, K. Tang and Z. Zhang}
}
\fi

\begin{document}

\maketitle

\begin{abstract}
Discrete tensor train decomposition is widely employed to mitigate the curse of dimensionality in solving high-dimensional PDEs through traditional methods. However, the direct application of the tensor train method  typically requires uniform grids of regular domains, which limits its application on non-uniform grids or irregular domains.
To address the limitation,
we develop a functional tensor train neural network (FTTNN) for solving high-dimensional
PDEs, which can represent PDE solutions on non-uniform grids or irregular domains. 
An essential ingredient of our approach is to represent the PDE solutions by the functional tensor train format whose TT-core functions are approximated by neural networks.
To give the functional tensor train representation, we propose and study functional tensor train rank and employ it into a physics-informed loss function for training. 
Because of tensor train representation, the resulting high-dimensional integral in the loss function can be computed via 
one-dimensional integrals by Gauss quadrature rules.
Numerical examples including high-dimensional PDEs on regular or irregular domains
are presented to demonstrate that the performance 
of the proposed FTTNN is better than that of 
Physics Informed Neural Networks (PINN).
\end{abstract}



\section{Introduction}
Partial Differential Equations (PDEs) are fundamental mathematical tools that describe
a wide range of natural and physical phenomena, such as fluid flow, heat transfer, 
electromagnetism, and quantum mechanics.
We are interested in high-dimensional PDEs, such as Schrödinger eigenvalue problems, which frequently arise in advanced scientific and engineering applications. 
Traditional numerical methods including finite difference methods, finite element methods, and spectral methods, have been well-established and widely used for solving PDEs in low-dimensional settings.
However, when applied to high-dimensional problems, these methods face significant challenges due to the curse of dimensionality (CoD).

In recent years, numerous numerical methods have been proposed to mitigate the CoD in high-dimensional PDEs. 
These methods can be broadly classified into two main categories: tensor-based methods \cite{khoromskij2018tensor,xiao2025provable} and neural network-based methods. The core idea of tensor-based methods is to tensorize high-dimensional discretized PDE operators due to their low-rank structures. Specifically, they apply discrete tensor decomposition methods to represent the discretized PDE operators. 
Common discrete tensor decomposition methods \cite{kolda2009tensor,oseledets2011tensor}, such as CANDECOMP/PARAFAC (CP) decomposition, Tucker decomposition, and tensor train (TT) decomposition, have been employed not only to reduce computational complexity but also to analyze the expressive power of deep learning models \cite{cohen2016expressive,cohen2016convolutional,khrulkov2018expressive}, including convolutional neural networks and recurrent neural networks.
Despite their advantages, their applicability to non-uniform grids  or irregular domains is limited.
On the other hand, neural network-based methods aim to search for an optimal solution within function space spanned by neural networks.
The most representative methods include Physics Informed Neural Networks (PINN) \cite{raissi2019physics}, 
Deep Galerkin Method (DGM) \cite{sirignano2018dgm}, and Deep Ritz Method (DRM) \cite{yu2018deep}. 
However, these methods rely on sampling methods to estimate high-dimensional integrals in loss functions, which can introduce statistical error and overshadow the capability of neural networks \cite{tang2023pinns}.

To address these issues, we develop a functional tensor train neural network (FTTNN) for solving high-dimensional PDEs.
This method introduces a regularization---functional tensor train decomposition (FTTD) to represent high-dimensional PDE solutions. Each TT-core function of FTTD is parameterized by a neural network. FTTD is a natural extension of discrete tensor train decomposition from discrete grids to the continuous domain. 
Meanwhile, we develop fundamental concepts including functional tensor train rank and functional tensor train representation.
The neural networks of FTTNN are trained by minimizing a physics-informed loss function, which involves a high-dimensional integral. For the high-dimensional integral, its tensor train structure allows us to estimate it by computing only a few one-dimensional integrals, instead of directly computing high-dimensional
integration. 
We also provide the approximation property of FTTNN.
Numerical experiments show that
compared with discrete tensor train decomposition, FTTNN is more flexible in approximating functions and is applicable to PDEs on irregular domains.
Additionally, the results illustrate that 
compared with PINN, FTTNN achieves better performance in terms of accuracy. 

There are some works on extending low-rank modeling for continuous function representation beyond meshgrids.
Oseledets \cite{oseledets2013constructive} gives explicit representations 
of several multivariate functions in the tensor-train format, such as the polynomial and sine functions. Gorodetsky et. al \cite{gorodetsky2019continuous} develop new algorithms and data structures for representing and computing with the functional tensor-train decomposition. Based on such a new continuous functional tensor train decomposition, some high-dimensional uncertainty quantification problems \cite{bigoni2016spectral} and multidimensional quantum dynamics \cite{soley2021functional} are studied.
 Luo et. al \cite{luo2023low} propose a low-rank tensor function representation parameterized by multilayer perceptrons for multi-dimensional data continuous representation. Novikov et. al \cite{novikov2021tensor} propose a continuous tensor train-based model for density estimation.
Wang et. al \cite{wang2024solving} propose tensor neural networks for solving high-dimensional boundary value problems and eigenvalue problems.

The differences between our work and the methods mentioned above are as follows. In the literature \cite{oseledets2013constructive,bigoni2016spectral,gorodetsky2019continuous,novikov2021tensor,soley2021functional}, the aim is to obtain approximations of solutions in the tensor train format, where its TT-core functions are parameterized by finite sum of basis functions.
Comparatively, our method intends to obtain the functional tensor train
approximation of the solutions, where its TT core functions are parameterized by neural networks, more suitable to capture complex PDE solutions.
In addition, we give the concept of functional tensor train rank, which establishes connections between discrete and continuous/functional tensor train representations for the solutions.
Compared with the Tucker decomposition-based work \cite{luo2023low}, FTTNN does not suffer the CoD. 
FTTNN also has stronger expressive power than the method proposed in \cite{wang2024solving},  as tensor train  decomposition has greater expressive efficiency than CP decomposition.

The rest of paper is organized as follows. In \Cref{section2}, we give the definition of functional tensor train rank and then introduce FTTNN.
How to compute a physics-informed loss function in the tensor train format and impose boundary conditions for training FTTNN is followed.
In \Cref{section3}, we give the approximation property of FTTNN.
In \Cref{section4}, numerical experiments are conducted to verify the accuracy of FTTNN.
Concluding remarks follow in \Cref{section5}.

\section{Methodology}\label{section2}
To start with, the details of the PDE model discussed in this paper are presented as follows.
Let $\Omega\subset \mathbb{R}^d$ denote a physical domain that is bounded, connected and with a polygonal boundary $\partial \Omega$, and $\mathbf{x} \in \Omega$ is  a physical variable.
The physics of the problem considered is governed by a PDE over the physical domain $\Omega$: find $u(\mathbf{x})$ such that
\begin{equation}\label{physical_problem}
	\begin{aligned}
		&\mathcal{L}u(\mathbf{x})=f(\mathbf{x}),\quad \forall \mathbf{x} \in \Omega,\\
		&\mathcal{B}u(\mathbf{x})=g(\mathbf{x}),\quad \forall \mathbf{x} \in \partial\Omega,
	\end{aligned}
\end{equation}
where $\mathcal{L}$ is a partial differential operator and $\mathcal{B}$ is a boundary operator, $f(\mathbf{x})$ is the source function, and  $g(\mathbf{x})$ specifies boundary conditions. 

Our goal is to solve \Cref{physical_problem} using neural network-based methods, particularly when the dimension $d$ is large, as traditional numerical methods become infeasible due to the curse of dimensionality.
Specifically, we represent $u(\mathbf{x})$ using a functional tensor train neural network $u(\mathbf{x};\theta)$ 
with trainable parameters $\theta$,
and $\theta$ is obtained by minimizing the following loss function,
\begin{equation}\label{loss_f}
    \mathrm{J}(\theta):=\int_{\Omega} (\mathcal{L}u(\mathbf{x};\theta)-f(\mathbf{x}))^2d\mathbf{x}+\beta \int_{\partial \Omega} (\mathcal{B}u(\mathbf{x};\theta)-g(\mathbf{x}))^2d\mathbf{x},
\end{equation}
where the first term (residual loss) and the second term (boundary loss) measure how well $u(\mathbf{x};\theta)$ satisfies the partial differential equations and the boundary conditions in \Cref{physical_problem}, respectively.
$\beta>0$ is a penalty parameter for boundary conditions. $\mathrm{J}(\theta)$ is a physics-informed loss function.
In the next section, we will explain how to represent $u(\mathbf{x})$ with the functional tensor train neural network.

\subsection{Functional tensor train Representation}
In the subsection, we first recall the discrete tensor train decomposition. Subsequently, we define the functional tensor train rank and then present the functional tensor train representation for $u$ in \Cref{physical_problem}. 
\begin{theorem}[Discrete tensor train decomposition]\label{TTD}
For a $d$-th order tensor $\mathrm{A}\in $$\mathbb{R}^{n_1\times \cdots \times n_d}$,
    there exists a tensor train decomposition \cite{oseledets2011tensor}
    \begin{align*}
        \mathrm{A}(i_1,\dots,i_d)=\sum_{\alpha_0=1}^{r_0}\sum_{\alpha_1=1}^{r_1}\cdots\sum_{\alpha_d=1}^{r_d} \mathrm{G_1}(\alpha_0,i_1,\alpha_1)\mathrm{G_2}(\alpha_1,i_2,\alpha_2)\cdots \mathrm{G_d}(\alpha_{d-1},i_d,\alpha_d),
    \end{align*}
     with TT-ranks $(r_0,r_1,\dots,r_d)$, where $r_0=r_d=1$, and for $k=1,\dots,d-1$,
    \begin{align*}
         r_k= \text{rank}(A_k), \ A_k(i_1,\dots,i_k;i_{k+1},\dots,i_d)=\mathrm{A}(i_1,\dots,i_d).
    \end{align*}
    Here, the first $k$ indices of $\mathrm{A}$ 
    enumerate the rows of the matrices $A_k$ and the last $d-k$ indices enumerate 
    the columns of $A_k$. The matrices $A_k$
    are called the unfoldings of the tensor $\mathrm{A}$.
\end{theorem}
\begin{definition}[Tensor function]
    Let $u(\cdot) : \Omega \subset \mathbb{R}^d \to \mathbb{R}$ be a bounded real-valued function, where $\Omega =\Omega_1\times  \cdots \times \Omega_d$,
    and each $\Omega_i \subset \mathbb{R}$ represents a definition domain in the $i$-th dimension, for $i=1,\dots,d$.
    When $\Omega$ is a discrete set, the function $u(\cdot)$ reduces to a discrete case, i.e., $d$-th order tensor.
    Therefore, $u(\cdot)$ can be regarded as a tensor function.
\end{definition}
\begin{definition}[Sampled tensor set]
    For a tensor function $u(\cdot): \Omega \subset \mathbb{R}^d \to \mathbb{R}$, the sampled tensor set $S[u]$ is defined as:
\begin{align*}
    S[u]:=&\{\mathrm{A}| \mathrm{A}\in \mathbb{R}^{n_1\times \cdots \times n_d}, \mathrm{A}{(i_1,\dots,i_d)}=u(\mathbf{x_1}{(i_1)},\dots,\mathbf{x_d}{(i_d)}),\\
    &
    \mathbf{x_1}\in \Omega_1^{n_1},\dots,\mathbf{x_d}\in \Omega_d^{n_d}, n_1,\dots,n_d \in \mathbb{N}_{+}
    \},
\end{align*}
where $\mathbf{x_i}$ denote the coordinate vector generated by sampling $n_i$ coordinates from  $\Omega_i$, and the collection of all such vectors $\mathbf{x_i}$ forms the set $\Omega_i^{n_i}$, for $i=1,\dots,d$.
\end{definition}
\begin{definition}[Functional tensor train rank]\label{FTT-rank}
    Given a tensor function $u :\Omega \subset \mathbb{R}^d  \to \mathbb{R}$, 
    we define the functional tensor train rank, denoted 
    by FTT-rank$[u]$, as the supremum of tensor train rank in the 
    sampled tensor set $S[u]$:
    \begin{align*}
        \text{FTT-rank}[u]:=(r_0,r_1,\dots,r_{d-1},r_d), r_k=\sup_{\mathrm{A}\in S[u]}\text{rank}(A_k), k=1,\dots,{d-1},
    \end{align*}
    and $r_0=r_d=1$.
\end{definition}
\begin{proposition}
    Let $\mathrm{A}\in \mathbb{R}^{n_1\times\cdots \times n_d}$
    be an arbitrary tensor. For $i=1,\dots,d$, define the discrete set $\Omega_i=\{1,2,\dots,n_i\}$, and set $\Omega=\Omega_1\times \cdots \times \Omega_d $.
    Define a tensor function $u: \Omega \to \mathbb{R}$ by $u(x_1,\dots,x_d)=\mathrm{A}(x_1,\dots,x_d)$, for any $(x_1,\dots,x_d)\in \Omega$.
    Then $\text{FTT-rank}[u]$ equals the TT-ranks of $\mathrm{A}$.
\end{proposition}
This proposition establishes the connection between the functional tensor train rank and the classical
tensor train rank in the discrete case.
\begin{theorem}[Functional tensor train decomposition]
    Let $u: \Omega \subset \mathbb{R}^d \to \mathbb{R}$ be a tensor function.
    If $\text{FTT-rank}[u]:=(r_0,r_1,\dots,r_{d-1},r_d)$ in Definition \ref{FTT-rank}, then there exist TT-core functions $u_i(\cdot): \Omega_i \to \mathbb{R}^{r_{i-1}\times r_{i}},i=1,\dots,d$,
    such that for any $\mathbf{x} \in \Omega$,
\begin{align}\label{hx}
u(\mathbf{x})&=u(x_1,\dots,x_d) \nonumber\\
&=\sum_{\alpha_0=1}^{r_0}\sum_{\alpha_1=1}^{r_1}\cdots\sum_{\alpha_d=1}^{r_d} u_1^{(\alpha_0,\alpha_1)}(x_1)u_2^{(\alpha_1,\alpha_2)}(x_2)\cdots u_d^{(\alpha_{d-1},\alpha_d)}(x_d).
\end{align}
\end{theorem}
This theorem is a natural extension of tensor train 
decomposition in Theorem \ref{TTD} from discrete grids to the continuous domain \cite{oseledets2013constructive}.
It can be observed that adjacent TT-core functions are linked by a common index. For instance, 
$\alpha_1$ establishes the connection between the first TT-core function $u_1$ and the second TT-core function $u_2$.

For the $i$-th TT-core function $u_i(x_i)$ in \Cref{hx}, a subnetwork $u_i(x_i;\theta_i)$ with parameters $\theta_i$ is employed for parameterization, for $i=1,\dots,d$. This derives the following functional tensor train neural network (FTTNN)
\begin{align}\label{FTNN}
    u(\mathbf{x};\theta)&= u(x_1,\dots,x_d;\theta) \nonumber\\
&=\sum_{\alpha_0=1}^{r_0}\sum_{\alpha_1=1}^{r_1}\cdots\sum_{\alpha_d=1}^{r_d} u_1^{(\alpha_0,\alpha_1)}(x_1;\theta_1)u_2^{(\alpha_1,\alpha_2)}(x_2;\theta_2)\cdots u_d^{(\alpha_{d-1},\alpha_d)}(x_d;\theta_d) \nonumber\\
&=u_1(x_1;\theta_1)u_2(x_2;\theta_2)\cdots u_d(x_d;\theta_d),
\end{align}    
where $\theta=\{\theta_1,\dots,\theta_d\}$ are the total trainable parameters from $d$ different subnetworks.

\subsection{Algorithm}\label{alg_tt}
Our goal is to solve \Cref{physical_problem}. This problem is reformulated as an optimization problem,  where we aim to minimize the loss function $\mathrm{J}(\theta)$ in \Cref{loss_f}, with $\theta$ being the parameters of FTTNN.
In this subsection, we provide computational details of the loss function $\mathrm{J}(\theta)$ for a class of PDEs. 

The boundary loss of $\mathrm{J}(\theta)$ ensures that
the FTTNN solution $u(\mathbf{x};\theta)$ in \Cref{FTNN} satisfies the boundary conditions.
The boundary conditions can also be satisfied by embedding them into the network architecture such that the boundary loss of $\mathrm{J}(\theta)$ is zero.
For example,
when $\Omega=\Omega_1 \times \cdots \times \Omega_d$ with each $\Omega_i=(a_i,b_i)\subset \mathbb{R}$ for $i=1,\dots,d$, 
the $i$-th subnetwork
$u_i(x_i;\theta_i)$ is constructed in the following way to satisfy homogeneous boundary conditions,
\begin{align}\label{bc_s}
    u_i(x_i;\theta_i):=(x_i-a_i)(b_i-x_i)\hat{u_i}(x_i;\theta_i),
\end{align}
where $\hat{u_i}(x_i;\theta_i) $ is a fully-connected neural network (FCN).

To simplify the derivation of computing the residual loss of $\mathrm{J}(\theta)$, we consider the following second-order elliptic operator:
\begin{align*}
    \mathcal{L}u(\mathbf{x};\theta):=-c_1 \Delta u(\mathbf{x};\theta)+b(\mathbf{x})u(\mathbf{x};\theta),
\end{align*}
where $c_1$ is a positive constant and $b(\mathbf{x})$ is a bounded function. Assume that the source function $f(\mathbf{x})$ in \Cref{physical_problem} and $b(\mathbf{x})$ can be represented by the functional tensor train format,
\begin{align*}
&f(\mathbf{x})=f(x_1,\dots,x_d)=f_1(x_1)f_2(x_2)\cdots f_d(x_d),\\
&b(\mathbf{x})=b(x_1,\dots,x_d)=b_1(x_1)b_2(x_2)\cdots b_d(x_d).
\end{align*}
When explicit tensor-train forms of 
$f(\mathbf{x})$ and $b(\mathbf{x})$ are not available, these functions can instead be approximated within the tensor-train formats.
Then the residual loss is written as 
\begin{align*}
    &\int_{\Omega} (\mathcal{L}u(\mathbf{x};\theta)-f(\mathbf{x}))^2d\mathbf{x}=\int_{\Omega}\Big(-c_1 \Delta u(\mathbf{x};\theta)+b(\mathbf{x})u(\mathbf{x};\theta)-f(\mathbf{x})\Big)^2 d\mathbf{x}\\
    &=\int_{\Omega}\Big(-c_1 \Delta u(\mathbf{x};\theta)+b(\mathbf{x})u(\mathbf{x};\theta)\Big)^2-2\Big(-c_1 \Delta u(\mathbf{x};\theta)+b(\mathbf{x})u(\mathbf{x};\theta)\Big)f(\mathbf{x})+f^2(\mathbf{x})d\mathbf{x},
\end{align*}
which is equivalent to minimize the following term
\begin{align*}
    &\int_{\Omega}\Big(-c_1 \Delta u(\mathbf{x};\theta)+b(\mathbf{x})u(\mathbf{x};\theta)\Big)^2-2\Big(-c_1 \Delta u(\mathbf{x};\theta)+b(\mathbf{x})u(\mathbf{x};\theta)\Big)f(\mathbf{x})d\mathbf{x}\\
    &=\underbrace{c_1^2\int_{\Omega}\Big(\Delta u(\mathbf{x};\theta)\Big)^2 d\mathbf{x}}_{I_1} +\underbrace{\int_{\Omega}b^2(\mathbf{x})u^2(\mathbf{x};\theta)d\mathbf{x}}_{I_2}\underbrace{-2c_1\int_{\Omega}b(\mathbf{x})\Delta u(\mathbf{x};\theta)u(\mathbf{x};\theta)d\mathbf{x}}_{I_3}\\
    &+\underbrace{2c_1\int_{\Omega}\Delta u(\mathbf{x};\theta)f(\mathbf{x})d\mathbf{x}}_{I_4}\underbrace{-2\int_{\Omega}b(\mathbf{x})f(\mathbf{x})u(\mathbf{x};\theta) d\mathbf{x}}_{I_5}\\
    &=I_1+I_2+I_3+I_4+I_5.
\end{align*}
Before calculating $I_1,I_2,I_3,I_4$, and $I_5$,  substituting \Cref{FTNN} into $\Delta u(\mathbf{x};\theta)$ yields,
\begin{align*}
    \Delta u(\mathbf{x};\theta)=&\underbrace{\frac{\partial^2 u_1(x_1;\theta_1)}{\partial x_1^2}u_2(x_2;\theta_2)\cdots u_d(x_d;\theta_d)}_{y_1}\\
&+\underbrace{u_1(x_1;\theta_1)\frac{\partial^2 u_2(x_2;\theta_2)}{\partial x_2^2}\cdots u_d(x_d;\theta_d)}_{y_2}\\
    &+\cdots
    +\underbrace{u_1(x_1;\theta_1)u_2(x_2;\theta_2)\cdots \frac{\partial^2 u_d(x_d;\theta_d)}{\partial x_d^2}}_{y_d}\\
    =&y_1+y_2+\cdots+y_d.
\end{align*}
$I_1$ is computed by
\begin{align*}
    I_1&=c_1^2\int_{\Omega}\Big(\Delta u(\mathbf{x};\theta)\Big)^2 d\mathbf{x}\\
    &=c_1^2\int_{\Omega}(y_1+y_2+\cdots+y_d)(y_1+y_2+\cdots+y_d) d\mathbf{x},
\end{align*}
where
\begin{align}
    &\int_{\Omega}y_1y_d d\mathbf{x} \nonumber\\
    =&\int_{\Omega}\Big(\frac{\partial^2 u_1(x_1;\theta_1)}{\partial x_1^2}u_2(x_2;\theta_2)\cdots u_d(x_d;\theta_d)\Big)\Big(u_1(x_1;\theta_1)u_2(x_2;\theta_2) \nonumber \\
    &\cdots \frac{\partial^2 u_d(x_d;\theta_d)}{\partial x_d^2}\Big)d\mathbf{x} \nonumber\\
=&\int_{\Omega} \Big(\frac{\partial^2 u_1(x_1;\theta_1)}{\partial x_1^2}\otimes u_1(x_1;\theta_1)\Big)
\Big(u_2(x_2;\theta_2)\otimes u_2(x_2;\theta_2)\Big) \nonumber\\
&\cdots 
\Big(u_d(x_d;\theta_d)\otimes\frac{\partial^2 u_d(x_d;\theta_d)}{\partial x_d^2}\Big)d\mathbf{x} \nonumber\\
=&\int_{\Omega_1}\Big(\frac{\partial^2 u_1(x_1;\theta_1)}{\partial x_1^2}\otimes u_1(x_1;\theta_1)\Big)dx_1
\int_{\Omega_2}\Big(u_2(x_2;\theta_2)\otimes u_2(x_2;\theta_2)\Big) dx_2 \nonumber\\
&\cdots 
\int_{\Omega_d} \Big(u_d(x_d;\theta_d)\otimes\frac{\partial^2 u_d(x_d;\theta_d)}{\partial x_d^2}\Big) dx_d \nonumber\\
\approx &\sum_{i=1}^{n_1} w_{i,1}\Big(\frac{\partial^2 u_1(x_{i,1};\theta_1)}{\partial x_{i,1}^2}\otimes u_1(x_{i,1};\theta_1)\Big)
\sum_{i=1}^{n_2} w_{i,2}\Big(u_2(x_{i,2};\theta_2)\otimes u_2(x_{i,2};\theta_2)\Big) \nonumber\\
&\cdots 
\sum_{i=1}^{n_d}w_{i,d}\Big(u_d(x_{i,d};\theta_d)\otimes\frac{\partial^2 u_d(x_{i,d};\theta_d)}{\partial x_{i,d}^2}\Big).\label{Gauss}
\end{align}
Here, we apply the Gauss-Legendre quadrature rule to compute $d$ one-dimensional integrals at the final step in \Cref{Gauss}. $x_{i,1},x_{i,2},\dots,x_{i,d}$ denote the quadrature points and $w_{i,1},w_{i,2},\dots,w_{i,d}$ are the quadrature weights. 
The remaining terms in $I_1$ can also be computed in the same way as $\int_{\Omega}y_1y_d d\mathbf{x}$.
$I_2$ means that
\begin{align*}
I_2=&\int_{\Omega}b^2(\mathbf{x})u^2(\mathbf{x};\theta)d\mathbf{x}\\
    =&\int_{\Omega_1}\Big(b_1(x_1)\otimes b_1(x_1)\otimes u_1(x_1;\theta_1)\otimes u_1(x_1;\theta_1)\Big)dx_1\\
&\cdots \int_{\Omega_d}\Big(b_d(x_d)\otimes b_d(x_d)\otimes u_d(x_d;\theta_d)\otimes u_d(x_d;\theta_d)\Big)dx_d.
\end{align*}
$I_3$ is represented as 
\begin{align*}
    I_3=&-2c_1\int_{\Omega} b(\mathbf{x})\Delta u(\mathbf{x};\theta)u(\mathbf{x};\theta)d\mathbf{x}\\
    =&-2c_1\int_{\Omega} b(\mathbf{x})(y_1+y_2+\cdots+y_d)u(\mathbf{x};\theta)d\mathbf{x},
\end{align*}
where 
\begin{align*}
&\int_{\Omega}b(\mathbf{x})y_1u(\mathbf{x};\theta)d\mathbf{x}=\int_{\Omega}b(\mathbf{x})\Big(\frac{\partial^2 u_1(x_1;\theta_1)}{\partial x_1^2}u_2(x_2;\theta_2)\cdots u_d(x_d;\theta_d)\Big)u(\mathbf{x};\theta)d\mathbf{x}\\
=&\int_{\Omega_1}b_1(x_1)\otimes \Big(\frac{\partial^2 u_1(x_1;\theta_1)}{\partial x_1^2}\Big)\otimes u_1(x_1;\theta_1)dx_1 \int_{\Omega_2}\Big(b_2(x_2)\otimes u_2(x_2;\theta_2)\\
&\otimes u_2(x_2;\theta_2)\Big) dx_2 
\cdots \int_{\Omega_d}b_d(x_d)\otimes u_d(x_d;\theta_d)\otimes u_d(x_d;\theta_d) dx_d,
\end{align*}
and the remaining terms in $I_3$
can also be represented in the TT format.
$I_4$ is written as 
\begin{align*}
    I_4=2c_1\int_{\Omega}\Delta u(\mathbf{x};\theta)f(\mathbf{x}) d\mathbf{x}=2c_1\int_{\Omega}(y_1+y_2+\cdots+y_d)f(\mathbf{x})d\mathbf{x},
\end{align*}
where 
\begin{align*}
    \int_{\Omega}y_1f(\mathbf{x})d\mathbf{x}=&
    \int_{\Omega_1}\Big(\frac{\partial^2 u_1(x_1;\theta_1)}{\partial x_1^2}\Big)\otimes h_1(x_1)dx_1 \int_{\Omega_2}u_2(x_2;\theta_2)\otimes h_2(x_2) dx_2 \\
    &\cdots \int_{\Omega_d}u_d(x_d;\theta_d)\otimes h_d(x_d) dx_d,
\end{align*}
and the remaining terms in $I_4$ can also be expressed in the TT format.
$I_5$ is calculated as 
\begin{align*}
    &I_5=-2\int_{\Omega} b(\mathbf{x})f(\mathbf{x})u(\mathbf{x};\theta) d\mathbf{x}\\
    =&-2\int_{\Omega_1}b_1(x_1)\otimes h_1(x_1)\otimes u_1(x_1;\theta_1)dx_1
    \int_{\Omega_2}b_2(x_2)\otimes h_2(x_2)\otimes u_2(x_2;\theta_2)dx_2 \\
    &\cdots \int_{\Omega_d}b_d(x_d)\otimes h_d(x_d)\otimes u_d(x_d;\theta_d)dx_d.
\end{align*}
It can be observed that $I_1$, $I_2$, $I_3$, $I_4$ and $I_5$ are represented in the TT format.
Leveraging this TT representation, these high-dimensional integrals can be efficiently computed by evaluating only one-dimensional integrals via Gauss–Legendre quadrature rules, analogous to the computation of $\int_{\Omega}y_1y_d d\mathbf{x}$ in \Cref{Gauss}.
Therefore, by FTTNN, we can efficiently computes the residual loss of $\mathrm{J}(\theta)$ without introducing any statistical error, as no sampling methods are employed in the calculation of the residual loss.

For more general operators $\mathcal{L}$, additional integrals are needed, such as 
for Schrödinger eigenvalue problems,
\begin{align*}
    \int_{\Omega} \nabla u \cdot \nabla u d\mathbf{x}=&\int_{\Omega}\Big(\frac{\partial u_1(x_1;\theta_1)}{\partial x_1} u_2(x_2;\theta_2)\cdots u_d(x_d;\theta_d)\Big)^2
    +\Big(u_1(x_1;\theta_1)\frac{\partial u_2(x_2;\theta_2)}{\partial x_2} \\
    &\cdots u_d(x_d;\theta_d)\Big)^2
    +\cdots+\Big(u_1(x_1;\theta_1)u_2(x_2;\theta_2) \cdots \frac{\partial u_d(x_d;\theta_d)}{\partial x_d}\Big)^2 d\mathbf{x},
\end{align*}
where 
\begin{align*}
    &\int_{\Omega}\Big(\frac{\partial u_1(x_1;\theta_1)}{\partial x_1} u_2(x_2;\theta_2)\cdots u_d(x_d;\theta_d)\Big)^2 d\mathbf{x}\\
=& \int_{\Omega_1} \frac{\partial u_1(x_1;\theta_1)}{\partial x_1}\otimes \frac{\partial u_1(x_1;\theta_1)}{\partial x_1} dx_1  
\int_{\Omega_2}  u_2(x_2;\theta_2) \otimes  u_2(x_2;\theta_2) dx_2\\
&\cdots 
\int_{\Omega_d}  u_d(x_d;\theta_d) \otimes  u_d(x_d;\theta_d)dx_d, 
\end{align*} 
and the remaining terms in $\int_{\Omega} \nabla u \cdot \nabla u d\mathbf{x}$ can also be written in the TT format. 
Hence, $\int_{\Omega} \nabla u \cdot \nabla u d\mathbf{x}$ can be represented in the TT format and then efficiently computed, like the calculation of $\int_{\Omega}y_1y_d d\mathbf{x}$ in \Cref{Gauss}. 

The training process of FTTNN is summarized in Algorithm \ref{alg}, where the output is the approximate PDE solution $u(\mathbf{x};\theta^*)$ with optimal parameters $\theta^*$ of FTTNN. The training is carried out in two stages: the first stage employs the Adam optimizer, followed by the LBFGS optimizer in the second stage.
\begin{algorithm}[!htb]
	\caption{Training the functional tensor train neural network (FTTNN)}
	\label{alg}
	\begin{algorithmic}[1]
		\REQUIRE FTT-ranks, dimension $d$, quadrature points $X$, quadrature weights $W$, learning rate for the Adam optimizer $\eta_a$, learning rate for the LBFGS optimizer $\eta_l$, maximum epoch number for Adam optimizer $E_a$, maximum epoch number for the LBFGS optimizer $E_l$.
         \STATE Use FTT-ranks for designing $d$ subnetworks of the functional tensor train network (see \Cref{FTNN}).
           \STATE Initialize parameters $\theta=(\theta_1,\dots,\theta_d)$ of the functional tensor train network.
           \FOR {$i = 1:E_a+E_l$}
           \STATE Write the loss function $\mathrm{J}(\theta)$ in the TT-format (see \Cref{alg_tt}).
          \STATE Estimate the loss function $\mathrm{J}(\theta)$ by using  $X$ and  $W$, similar to the calculation of $\int_{\Omega}y_1y_d d\mathbf{x}$ in \Cref{Gauss}.
          \STATE  Compute the gradients $\nabla_{\theta} \mathrm{J}(\theta)$.
            \IF{ $i>E_a$}
            \STATE  Update the parameters $\theta$ by LBFGS optimizer with learning rate $\eta_l$.
            \ELSE
             \STATE  Update the parameters $\theta$ by Adam optimizer with learning rate $\eta_a$.
            \ENDIF
		\ENDFOR
         \STATE Let $\theta^*=\theta$, where $\theta$ includes the parameters of the functional tensor train network at the last epoch.
		\ENSURE The approximate  PDE solution $u(\mathbf{x};\theta^*)$.
	\end{algorithmic}
\end{algorithm}
\section{Theoretical analysis}\label{section3}
In this section, we follow \cite{wang2024solving} 
to illustrate the approximation property of FTTNN.

\begin{theorem}[Approximation error]
    Assume $\Omega=\Omega_1 \times \Omega_2 \times \cdots \times \Omega_d$ with $\Omega_i\subset \mathbb{R}$, for $i=1,\dots,d$, and a tensor function $u(\mathbf{x})\in  L^2(\Omega)$.
    Then for any tolerance $\epsilon>0$, there exist functional tensor train rank $(r_0,r_1,\dots,r_d)$
    and the corresponding functional tensor train neural network $u(\mathbf{x};\theta)$ defined by \Cref{FTNN} such that the following approximation property holds,
    \begin{align*}
        \Vert u(\mathbf{x})-u(\mathbf{x};\theta) \Vert_{L^2(\Omega)} \leq \epsilon.
    \end{align*}
\end{theorem}
\begin{proof}
According to the isomorphism relation $L^2(\Omega)\cong L^2(\Omega_1)\otimes \cdots \otimes L^2(\Omega_d) $, for any $\epsilon_1$, there exists 
a positive integer $r$ such that
\begin{align}\label{p1}
    \Vert u(\mathbf{x})-h(\mathbf{x})\Vert_{L^2(\Omega)}<\epsilon_1,
\end{align}
where $h(\mathbf{x})=\sum_{\alpha=1}^{r}h_1^{\alpha}(x_1)\cdots h_d^{\alpha}(x_d), h_i^{\alpha}(x_i)\in L^2(\Omega_i), \alpha=1,\dots, r, i=1,\dots,d$. $h(\mathbf{x})$ can be rewritten in the functional tensor train format by using Kronecker delta,
\begin{align*}
&h(\mathbf{x})=\sum_{\alpha_1=1}^{r_1}\sum_{\alpha_2=1}^{r_2}\cdots \sum_{\alpha_{d-1}=1}^{r_{d-1}} h_1^{\alpha_1}(x_1)\delta(\alpha_1,\alpha_2)h_2^{\alpha_2}(x_2)\cdots \delta(\alpha_{d-2},\alpha_{d-1})h_d^{\alpha_{d-1}}(x_d)\\
&=\sum_{\alpha_0=1}^{r_0}\sum_{\alpha_1=1}^{r_1}\sum_{\alpha_2=1}^{r_2}\cdots \sum_{\alpha_{d-1}=1}^{r_{d-1}}\sum_{\alpha_d=1}^{r_d} h_1^{(\alpha_0,\alpha_1)}(x_1)\underbrace{\delta(\alpha_1,\alpha_2)h_2^{\alpha_2}(x_2)}_{h_2^{(\alpha_1,\alpha_2)}(x_2)}\cdots h_d^{(\alpha_{d-1},\alpha_d)}(x_d)\\
&=\sum_{\alpha_0=1}^{r_0}\sum_{\alpha_1=1}^{r_1}\sum_{\alpha_2=1}^{r_2}\cdots \sum_{\alpha_{d-1}=1}^{r_{d-1}}\sum_{\alpha_d=1}^{r_d} h_1^{(\alpha_0,\alpha_1)}(x_1)h_2^{(\alpha_1,\alpha_2)}(x_2)\cdots   h_d^{(\alpha_{d-1},\alpha_d)}(x_d),
\end{align*}
where we choose $r_1,\dots,r_{d-1}$ such that $r=\min\{r_1,\dots,r_{d-1}\}$.

According to \cite{evans2022partial}, there exists a continuous function such that $h_i^{(\alpha_{i-1},\alpha_i)}(x_i)$, for $i=1,\dots,d$, can be approximated with arbitrary accuracy under the norm $L^2$. In addition,  \cite{hornik1991approximation,leshno1993multilayer} show that FCN with at least one hidden layer and a non-polynomial activation function can approximate arbitrary continuous functions on a compact set under the norm $L^2$. 
There exist FCN $u_{i}^{(\alpha_{i-1},\alpha_i)}(x_i;\theta_i),i=1,\dots,d$, such that for any $\delta>0$,
\begin{align}
    \Vert h_i^{(\alpha_{i-1},\alpha_i)}(x_i)-u_{i}^{(\alpha_{i-1},\alpha_i)}(x_i;\theta_i)\Vert_{L^2(\Omega_i)}< \delta. \label{p2}
\end{align}
Here, $\theta_i$ are the parameters of FCN.

For any $g_i(x_i)\in L^2(\Omega_i),i=1,\dots,d$, 
\begin{align}\label{s_inequality}
    \Big \Vert \prod_{i=1}^d g_i(x_i)\Big \Vert_{L^2(\Omega)}=\prod_{i=1}^{d}\Vert g_i(x_i) \Vert_{L^2(\Omega_i)}^2.
\end{align}

Denote $e^{(\alpha_{i-1},\alpha_i)}_i(x_i):=h_i^{(\alpha_{i-1},\alpha_i)}(x_i)-u_{i}^{(\alpha_{i-1},\alpha_i)}(x_i;\theta_i),\ i=1,\dots, d$, $M:=\max\limits_i  \Vert h_i^{(\alpha_{i-1},\alpha_i)}(x_i)\Vert _{L^2(\Omega_i)}$, $\boldsymbol{\alpha}:=\{\alpha_0,\dots,\alpha_d\}$. We apply the triangle inequality, \Cref{p1}, \Cref{p2} and \Cref{s_inequality} to 
\begin{align*}
    &\Vert h(\mathbf{x})-u(\mathbf{x};\theta)\Vert_{L^2(\Omega)}\\
    =&\Bigg\Vert \sum_{\boldsymbol{\alpha}} h_1^{(\alpha_0,\alpha_1)}(x_1)\cdots   h_d^{(\alpha_{d-1},\alpha_d)}(x_d)- \sum_{\boldsymbol{\alpha}} \Big(h_1^{(\alpha_0,\alpha_1)}(x_1)-e^{(\alpha_0,\alpha_1)}_1(x_1)\Big)\\
    &\cdots   \Big(h_d^{(\alpha_{d-1},\alpha_d)}(x_d)-e^{(\alpha_{d-1},\alpha_d)}_d(x_d)\Big)\Bigg\Vert_{L^2(\Omega)}\\
\leq& \sum_{\boldsymbol{\alpha}} \Bigg\Vert h_1^{(\alpha_0,\alpha_1)}(x_1)\cdots   h_d^{(\alpha_{d-1},\alpha_d)}(x_d)-\Big(h_1^{(\alpha_0,\alpha_1)}(x_1)-e^{(\alpha_0,\alpha_1)}_1(x_1)\Big)\\
&\cdots   \Big(h_d^{(\alpha_{d-1},\alpha_d)}(x_d)-e^{(\alpha_{d-1},\alpha_d)}_d(x_d)\Big) \Bigg\Vert_{L^2(\Omega)}\\
\leq &\sum_{\boldsymbol{\alpha}}((M+\delta)^d-M)\\
\leq & R^{d-1}\Big((M+\delta)^d-M\Big),
\end{align*}
where $R=\max\{r_0,r_1,\dots,r_d\}$.
Based on triangle inequality, 
\begin{align*}
    \Vert u(\mathbf{x})-u(\mathbf{x};\theta)\Vert_{L^2(\Omega)}
    &\leq \Vert u(\mathbf{x})-h(\mathbf{x})\Vert_{L^2(\Omega)}+\Vert h(\mathbf{x})-u(\mathbf{x};\theta)\Vert_{L^2(\Omega)}\\
    &\leq \epsilon_1+ R^{d-1}\Big((M+\delta)^d-M^d\Big)\leq \epsilon/2+\epsilon/2=\epsilon.
\end{align*}
The proof has been completed.
\end{proof}

\section{Numerical experiments}\label{section4}
In this section, five experiments are conducted to demonstrate the accuracy of FTTNN.
In test problem 1, FTTNN is compared with discrete tensor train decomposition for function approximation.
In test problem 2, we solve three-dimensional Poisson equations on irregular domains.
Additionally, we solve high-dimensional Poisson equations, high-dimensional Helmholtz equations, and Schrödinger eigenvalue problems on regular domains in test problem 2--5, respectively.
To evaluate the accuracy of FTTNN, the relative errors are defined by
\begin{equation}\label{rl}
    \text{Relative error}=\frac{\Vert u(\mathbf{x};\theta)-u(\mathbf{x}) \Vert_2}{\Vert u(\mathbf{x}) \Vert_2},
\end{equation}
where $u(\mathbf{x};\theta)$ 
is the solution approximated by FTTNN, and $u(\mathbf{x})$ is the exact solution. 

In \Cref{alg}, the functional tensor train rank is set to $\text{FTT-rank}[u]=(1,r,\dots,r,1)$.
Each subnetwork of FTTNN and the neural network of PINN are both FCNs with one hidden layer.
All neural networks employ the sine function as their activation function and are trained exclusively on a single NVIDIA A100 Tensor Core GPU.

\subsection{Test problem 1: Function approximation}
Approximating the following function with a singularity is tested:
\begin{align*}
    u(\mathbf{x})=\frac{1}{\sqrt{\sum_{i=1}^d x_i^2+10^{-12}}}, \quad \mathbf{x} \in [-1,1]^d.
\end{align*}

We consider two cases with $d=4$ and $d=6$.
The hidden layer of FTTNN contains 30 hidden neurons. 
To train FTTNN, we sample 100,000 points uniformly from the $d$-dimensional hypercube $[-1,1]^d$ (denoted by $\mathbf{x}\sim U([-1,1]^d)$)
and generate their corresponding labels $u(\mathbf{x})$. This forms the training dataset $\{{\mathbf{x}}^{(i)},u({\mathbf{x}}^{(i)})\}_{i=1}^{100000}$.
The training process employs the Adam optimizer with a learning rate of 0.003 for the initial 5,000 epochs. Subsequently, we switch to the L-BFGS optimizer with a learning rate of 0.1 for an additional 1,000 epochs. For comparison, we also use the discrete tensor train decomposition (TTD) to approximate the function $u(\mathbf{x})$. In this approach, we select $15\%$ points from a $30^d$ uniform grid as the training data points. Based on these training data points, we apply the tensor recovery method using the TTD to recover the function $u(\mathbf{x})$ on the remaining points of the grid. The TT-cores of TTD are estimated using the TT-ALS algorithm \cite{steinlechner2016riemannian}. 

\Cref{Com1} shows the number of parameters and relative errors of FTTNN and TTD in the case of  $d=4,r=2$ or $d=6,r=3$. 
The relative errors are computed through \Cref{rl} using the approximate solution and the exact solution. 
The parameters of FTTNN include weights and biases, while those of TTD consist of all the TT-cores. 
Although FTTNN has more parameters than TTD, its relative error is significantly smaller. 
This improvement can be attributed to the fact that FTTNN is not restricted to uniform grids. 
Actually, we need finer meshes near the singularities for more accurate approximation. This is difficult to achieve with discrete tensor train decomposition, 
and highlights the flexibility of FTTNN.
\begin{table}[htbp]
\caption{Comparison FTTNN with discrete tensor train decomposition (TTD), test problem 1.}\label{Com1}
\begin{center}
  \begin{tabular}{|c|c|c|c|c|} \hline
    Method & $d$ & $r$ &Number of parameters&Relative error  \\ \hline
    TTD& 4  & 2&360&$4.1\times 10^{-1}$ \\
        FTTNN& 4 &2&4332& $4.9\times 10^{-2}$ \\
        TTD& 6 & 3&5400& $2.2\times 10^{-1}$\\
        FTTNN& 6 &3&7242&$3.6\times 10^{-3}$ \\\hline
  \end{tabular}
\end{center}
\end{table}

\subsection{Test problem 2: Three-dimensional Poisson equation on irregular domains}
We consider the three-dimensonal Poisson equation on L shape domain $\Omega=\Omega_0 \times (-1,1)$ where $\Omega_0=(-1,1)^2\setminus ([0,1)\times (-1,0])$
\begin{equation*}
    -\Delta u(\mathbf{x})=f(\mathbf{x}), \quad \mathbf{x} \in \Omega,
\end{equation*}
where the source fucntion $f(\mathbf{x})=6x_1(x_2-x_2^3)(1-x_3^2)+6x_2(x_1-x_1^3)(1-x_3^2)+2(x_2-x_2^3)(x_1-x_1^3)$, and homogeneous Dirichlet boundary conditions are imposed on $\partial \Omega$. Its exact solution is 
$
    u(\mathbf{x})=(x_1-x_1^3)(x_2-x_2^3)(1-x_3^2)
$.

The hidden layer of FTTNN contains 50 neurons. FTTNN is compared with PINN, whose hidden layer is configured with 90 neurons to match the total number of parameters of FTTNN. The parameters of FTTNN and PINN are obtained by minimizing the loss function:
\begin{align}
    \mathrm{J}(\theta)=&\int_{-1}^1\int_{-1}^1\int_{-1}^0 (-\Delta u(\mathbf{x};\theta)-f)^2 dx_1 dx_2 dx_3 \nonumber\\
    &+\int_{-1}^1\int_0^1\int_0^1 (-\Delta u(\mathbf{x};\theta)-f)^2 dx_1 dx_2 dx_3+\beta \int_{\partial \Omega} (u(\mathbf{x};\theta)-0)^2 d\mathbf{x}, \label{loss_L}
\end{align}
where the first two terms correspond to the residual loss, and the last term represents the boundary loss.
We set $\beta=100, r=2$ in \Cref{loss_L}.
To efficiently compute the residual loss in FTTNN,
$(-1,0)$ and $(0,1)$ are each divided into 20 subintervals, with 20 quadrature points selected within each subinterval. Additionally, $(-1,1)$ is partitioned into 40 subintervals, with 20 quadrature points chosen in each subinterval.
The boundary loss in FTTNN is estimated by the rectangle rule, where  $\partial \Omega$ is equally divided into subdomains with $\Delta x_1=\Delta x_2=\Delta x_3=1/10$.
For PINN, its residual loss is computed using 12,000 points generated by uniformly sampling from $\Omega$, and its boundary loss is calculated in the same manner as in FTTNN.
In \Cref{alg}, the training hyperparameter are specified as follows: the Adam optimizer runs for a maximum of $E_{a}=5000$ epochs with a learning rate $\eta_{a}=0.003$
, and the LBFGS optimizer is limited to $E_{l}=1000$ epochs with a learning rate $\eta_{l}=0.1$.
PINN is trained following the same procedure.

As shown in \Cref{L_relative}, the relative errors computed by FTTNN and PINN indicate that FTTNN attains a lower relative error than PINN. \Cref{plotLshape}(a) is the exact solution at $x_3=0.5$ for this test problem, and \Cref{plotLshape}(b) shows the solution at $x_3=0.5$ obtained by FTTNN, where it can be seen that they are visually indistinguishable. Their difference $u(\mathbf{x};\theta)-u$ is shown in \Cref{plotLshape}(c), clearly demonstrating that the pointwise error is small.
The results imply that FTTNN can accurately solve PDEs on irregular domains, but discrete tensor train decomposition is not directly applicable to such PDEs. 
\begin{table}[htbp]
 \caption{Relative error, test problem 2.}
    \label{L_relative}
\begin{center}
  \begin{tabular}{|c|c|c|} \hline
    d &  FTTNN &PINN  \\ \hline
3 & $7.2\times 10^{-4}$  &$1.3\times 10^{-3}$\\ \hline
  \end{tabular}
\end{center}
\end{table}
\begin{figure}[!htb]
	\centering
     \subfloat[][The exact solution $u(\mathbf{x})$]{\includegraphics[width=.3\textwidth]{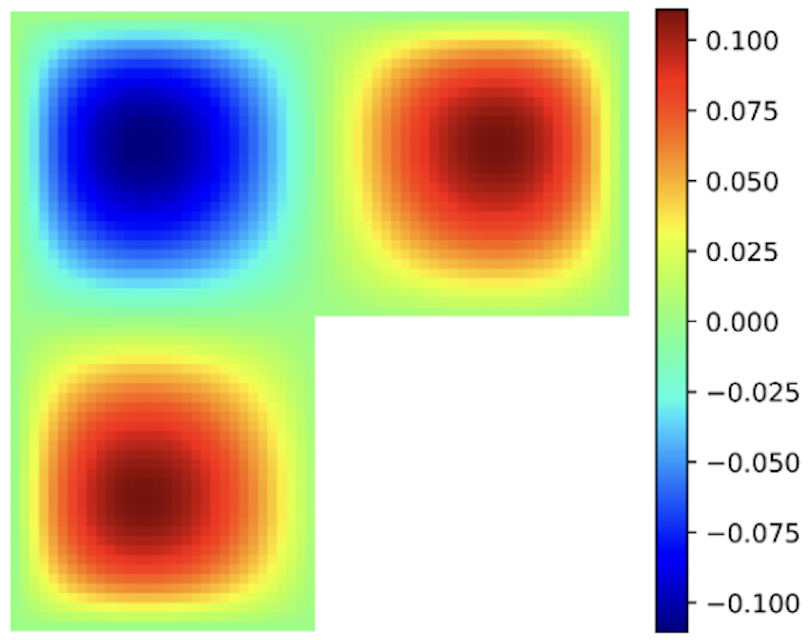}}
     \subfloat[][FTTNN solution $u(\mathbf{x};\theta)$]{\includegraphics[width=.3\textwidth]{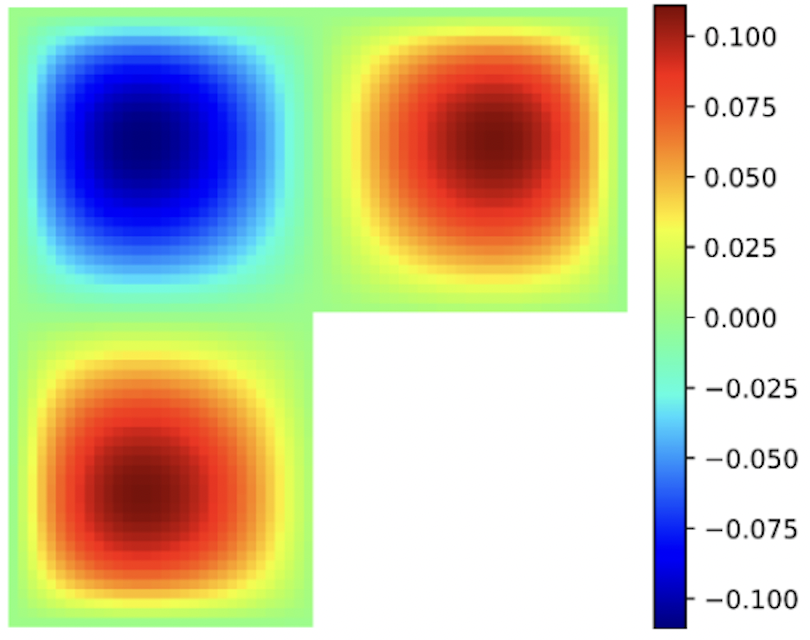}}
     \subfloat[][$u(\mathbf{x};\theta)-u$]{\includegraphics[width=.3\textwidth]{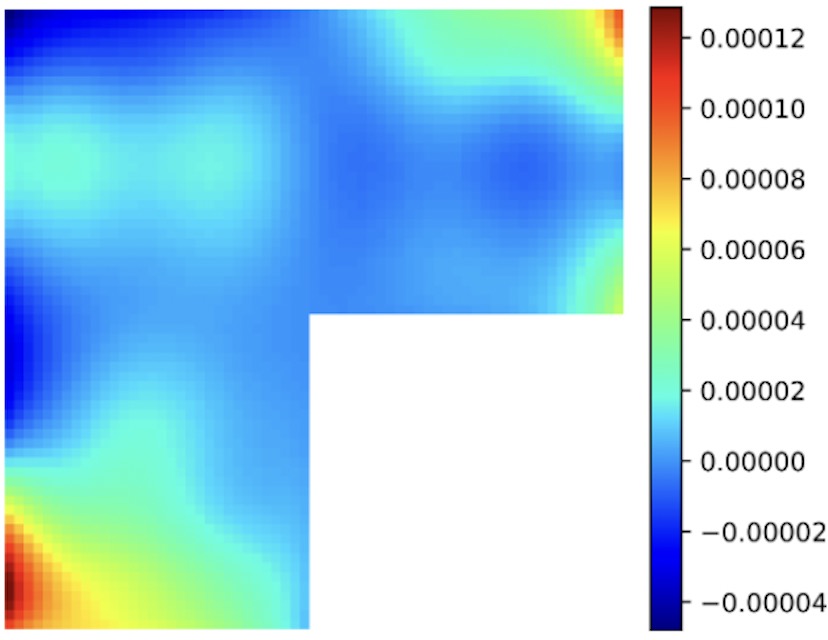}}
	\caption{Solutions of three-dimensional Poisson equations on L shape domain at $x_3=0.5$, test problem 2.}
    \label{plotLshape}
\end{figure}

\subsection{Test problem 3: High-dimensional Poisson equation on regular domains}\label{t3_section}
The high-dimensional Poisson equation with homogeneous Dirichlet boundary condition is considered,
\begin{equation*}
    -\Delta u(\mathbf{x})=(d+3)\pi^2\sum_{k=1}^d\sin(2\pi x_k)\prod_{i\neq k}^d \sin(\pi x_i),\quad \mathbf{x} \in (-1,1)^d,
\end{equation*}
whose exact solution is given by
$   u(\mathbf{x})=\sum_{k=1}^d\sin(2\pi x_k)\prod_{i\neq k}^d \sin(\pi x_i)
$.

Three cases ($d=3,5,7$) are considered.
FTTNN is compared with PINN, whose hidden layer consists of 100 neurons.
For the same total number of parameters in both PINN and FTTNN, the hidden layer size and the ranks of FTTNN for different dimensions are provided in \Cref{setting1}.
To efficiently compute the residual loss in FTTNN,
we decompose each $\Omega_i, i=1,\dots,d$ into 30 equal subintervals and 
choose 30 quadrature points on each subinterval. For PINN, its residual loss is evaluated using 5,000, 10,000, and 20,000 sample points for $d=3,5$, and $7$, respectively.
The boundary conditions for both FTTNN and PINN are enforced using a hard constraint, as described in \Cref{bc_s}.
The hyperparameters for training FTTNN and PINN are the same as those of test problem 2.
\begin{table}[htbp]
\caption{The settings for TTNN, test problem 3.}\label{setting1}
\begin{center}
  \begin{tabular}{|c|c|c|} \hline
    $d$ &   $r$ & hidden layer size \\ \hline
    3 & 2  &50\\
        5 & 3 &40\\
        7 &  4 &35\\ \hline
  \end{tabular}
\end{center}
\end{table}

\Cref{test1_error} gives the relative errors of FTTNN and PINN in terms of $d$. For $d=3$, FTTNN and PINN yield comparable relative errors. As the dimension increases to 
$d=5$, the relative error of FTTNN becomes one order of magnitude smaller than that of PINN. With a further increase in dimension to $d=7$, the relative error of FTTNN is reduced by two orders of magnitude compared with that of PINN. The relative errors of FTTNN and PINN during the training process ($d=3$) are shown in \Cref{rel1_plot}. As the number of epochs increases, it is clear that the relative errors of FTTNN and PINN both converge to a small value, and the relative errors of FTTNN decrease faster than that of PINN.  
\Cref{problem3_plot} also illustrates the solutions at $x_3=0.1$ estimated by FTTNN and PINN ($d=3$). From \Cref{problem3_plot}(b) and \Cref{problem3_plot}(c), it can be seen that the estimated solutions computed by FTTNN and PINN both agree with the exact solution (see \Cref{problem3_plot}(a)).
\begin{table}[htbp]
\caption{Relative error, test problem 3.}
    \label{test1_error}
\begin{center}
  \begin{tabular}{|c|c|c|} \hline
   $d$ &   FTTNN &PINN  \\ \hline
    3 & $2.6\times 10^{-5}$  &$5.5\times 10^{-5}$\\
        5 & $2.5\times 10^{-4}$ &$2.0\times 10^{-3}$\\
        7 &  $1.3\times 10^{-4}$ &$8.0\times 10^{-2}$\\ \hline
  \end{tabular}
\end{center}
\end{table}

\begin{figure}[!htb]
    \centering
\includegraphics[width=0.7\linewidth]{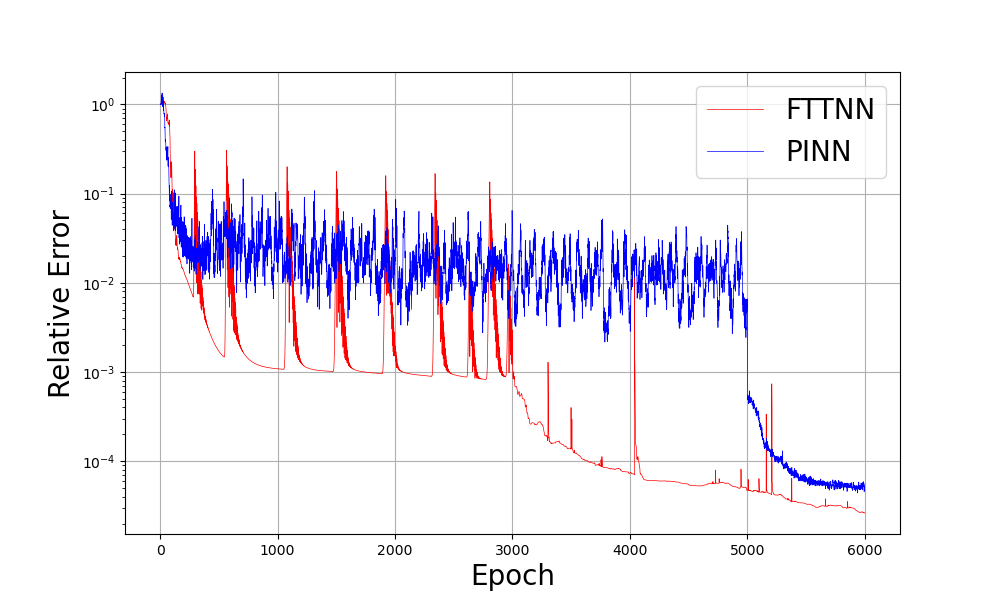}
     \caption{Relative error vs epoch ($ d=3, r=2$), test problem 3: }
    \label{rel1_plot}
\end{figure}

\begin{figure}[!htb]
	\centering
     \subfloat[][The exact solution]{\includegraphics[width=.3\textwidth]{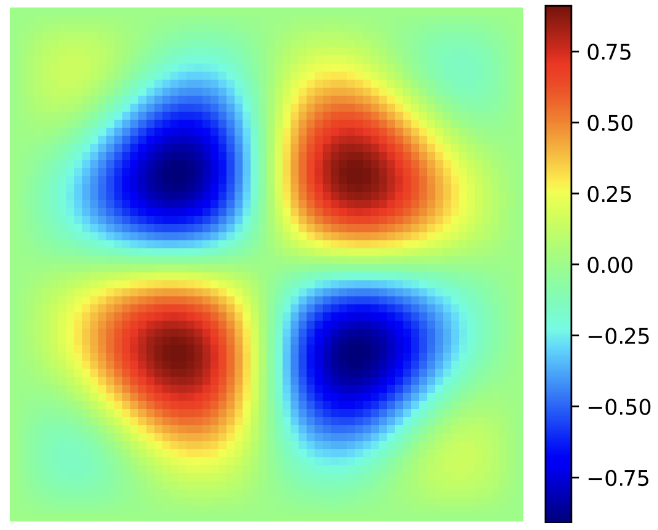}}
     \subfloat[][FTTNN solution]{\includegraphics[width=.3\textwidth]{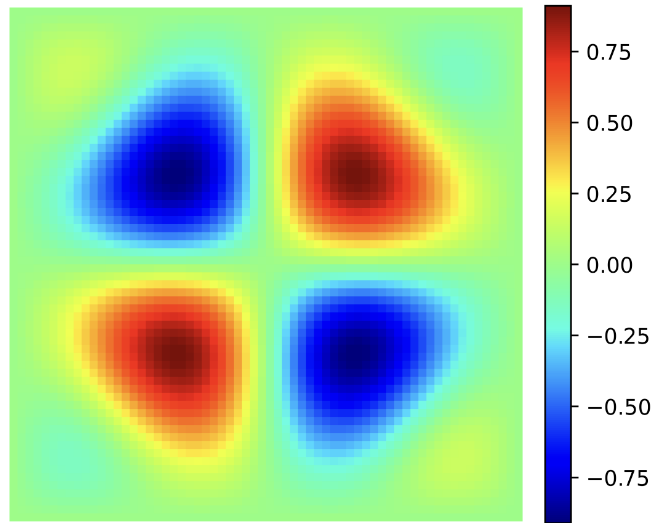}}
     \subfloat[][PINN solution]{\includegraphics[width=.3\textwidth]{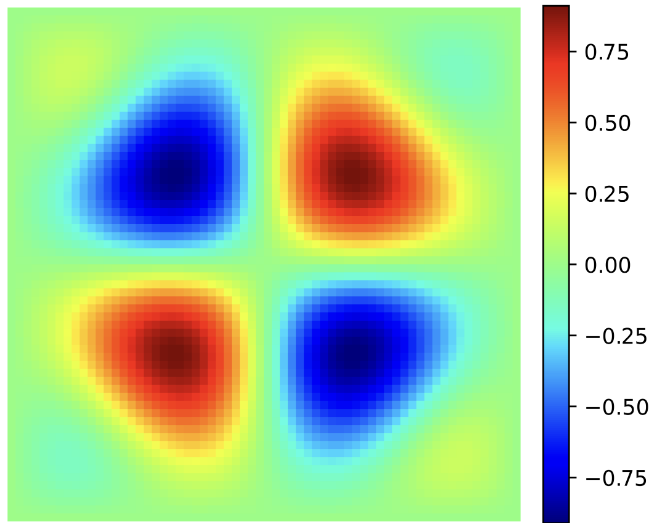}}
	\caption{Solutions of three-dimensional Poisson equations on regular domains at $x_3=0.1$, test problem 3.}
    \label{problem3_plot}
\end{figure}

\subsection{Test problem 4: High-dimensional Helmholtz equation}
\subsubsection{High-dimensional Helmholtz equation}\label{h1}
We try to solve the following Helmholtz
equation with homogeneous Dirichlet boundary condition,
\begin{equation*}
    -\Delta u(\mathbf{x})-u(\mathbf{x})=(4\pi^2d-1)\prod_{k=1}^d \sin(2\pi x_k),\quad \mathbf{x} \in (0,1)^d,
\end{equation*}
whose exact solution is 
$
    u(\mathbf{x})=\prod_{k=1}^d \sin(2\pi x_k)
$.

Two cases ($d=3,5$) are considered. 
FTTNN is compared with PINN, whose hidden layer contains 100 neurons. To ensure a comparable number of parameters, the hidden layer size and the ranks of the FTTNN are specified in \Cref{setting2}.
For efficiently computing the residual loss, each $\Omega_i,i=1,\dots,d$ is divided into 40 equal intervals and then choose 40 quadrature points on each subinterval.
For computing the residual loss in PINN with $d=3,5$, we generate 5000 points and 10000 points by uniformly sampling from $\Omega$, respectively.
The boundary conditions for both FTTNN and PINN are enforced in accordance with test problem 3.
The hyperparameters for training FTTNN and PINN are the same as those of test problem 3.
\begin{table}[htbp]
\caption{The settings for TTNN, test problem \ref{h1}.}
    \label{setting2}
\begin{center}
  \begin{tabular}{|c|c|c|} \hline
   $d$ &   $r$ & hidden layer size   \\ \hline
    3 & 2  &40\\
        5 & 3 &50\\ \hline
  \end{tabular}
\end{center}
\end{table}

\Cref{test2_relative} shows the relative errors computed by FTTNN and PINN. Although the relative error of FTTNN is a little larger than that of PINN in the case of $d=3$, its relative error becomes one order of magnitude smaller than the relative error of PINN as the dimension increases to $d=5$.
The relative errors of FTTNN and PINN during the training process ($d=5$) are shown in \Cref{plot2}, where it is obvious that the relative error of FTTNN decreases faster than that of PINN and finally converges to $1.3\times 10^{-4}$.
Although the relative error of FTTNN exhibits large fluctuations during the training process with the Adam optimizer, a learning rate decay strategy can be adopted to alleviate the issue.
\Cref{problem41_plot} also provides the solutions computed by FTTNN and PINN in the case of $d=5$ and $x_3=x_4=x_5=0.1$. In \Cref{problem41_plot}(a), the exact solution for this test problem is given. The estimated FTTNN solution shown in \Cref{problem41_plot}(b) is consistent with the exact solution, but the estimated PINN solution (denoted by $\hat{u}(\mathbf{x})$) shown in \Cref{problem41_plot}(e) slightly differs from the exact solution. In addition, the difference $u(\mathbf{x};\theta)-u(\mathbf{x})$ of FTTNN is far less than the difference $\hat{u}(\mathbf{x})-u(\mathbf{x})$ of PINN.
\begin{table}[htbp]
 \caption{Relative error, test problem \ref{h1}.}
    \label{test2_relative}
\begin{center}
  \begin{tabular}{|c|c|c|} \hline
    $d$ &   FTTNN &PINN  \\ \hline
3 & $8.0\times 10^{-5}$  &$4.0\times 10^{-5}$\\
        5 &  $1.3\times 10^{-4}$&$1.5\times 10^{-2}$\\ \hline
  \end{tabular}
\end{center}
\end{table}

\begin{figure}[!htb]
    \centering
    \includegraphics[width=0.7\linewidth]{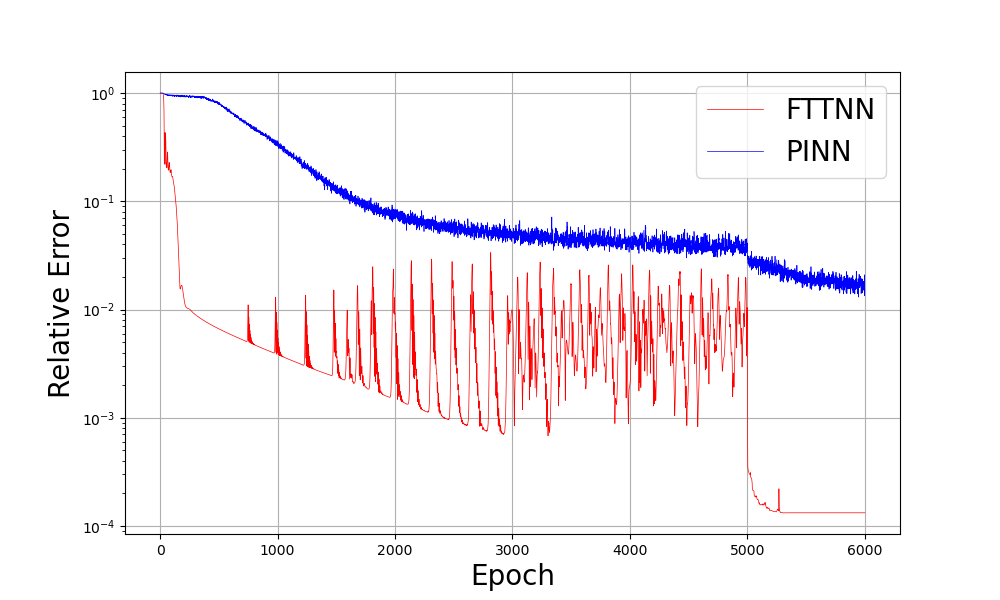}
     \caption{Relative error vs epoch ($d=5, r=3$), test problem \ref{h1}}
     \label{plot2}
\end{figure}
\begin{figure}[!htb]
	\centering
     \subfloat[][The exact solution]{\includegraphics[width=.3\textwidth]{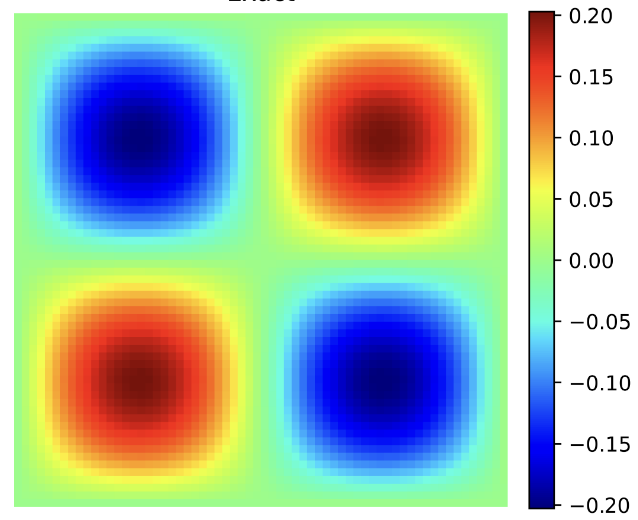}}
     \subfloat[][FTTNN solution $u(\mathbf{x};\theta)$]{\includegraphics[width=.3\textwidth]{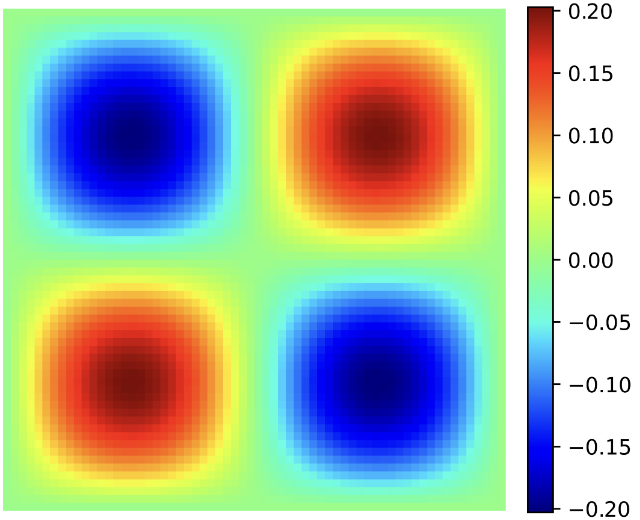}}
     \subfloat[][$u(\mathbf{x};\theta)-u(\mathbf{x})$]{\includegraphics[width=.3\textwidth]{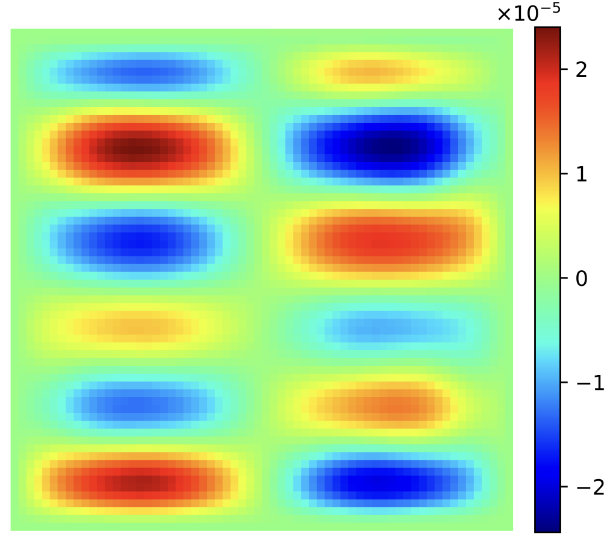}}\\
     \subfloat[][The exact solution]{\includegraphics[width=.3\textwidth]{exact_p41_5d.png}}
     \subfloat[][PINN solution $\hat{u}(\mathbf{x})$]{\includegraphics[width=.3\textwidth]{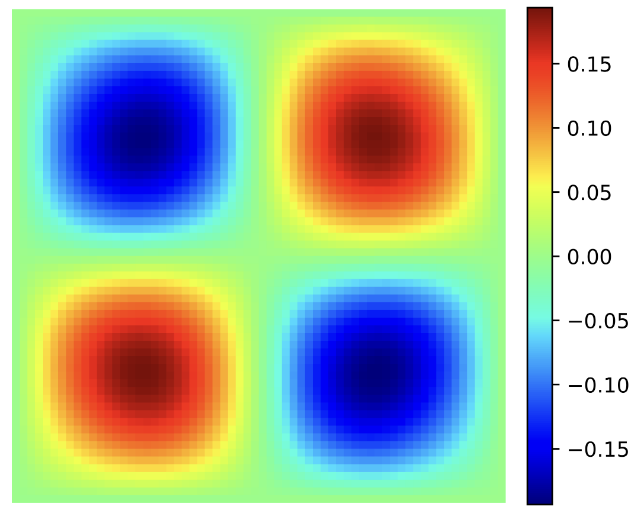}}
     \subfloat[][$\hat{u}(\mathbf{x})-u(\mathbf{x})$]{\includegraphics[width=.3\textwidth]{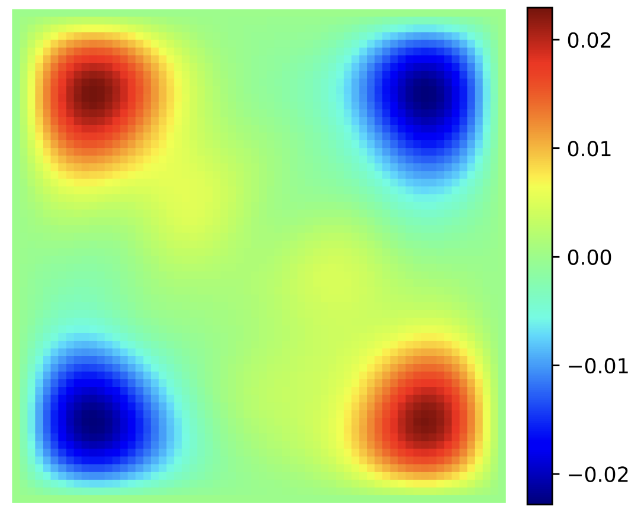}}
	\caption{Solutions of five-dimensional Helmoltz equations at $x_3=x_4=x_5=0.1$, test problem \ref{h1}.}
    \label{problem41_plot}
\end{figure}

\subsubsection{Three-dimensional Helmholtz equation with high wave number}\label{high}
In this test problem, three-dimensional Helmholtz equation with high wave number $k$ is considered,
\begin{equation*}
    -\Delta u(\mathbf{x})-k^2u(\mathbf{x})=2\sin(kx_1)\sin(kx_2)\sin(kx_3),\quad \mathbf{x} \in (0,1)^3,
\end{equation*}
with Dirichlet boundary conditions on all faces of a unit cubic. Its exact solution is
$
    u(\mathbf{x})=\sin(kx_1)\sin(kx_2)\sin(kx_3)/k^2
$.
Solving such equations is challenging due to the emergence of highly oscillatory solutions, which requires very fine meshes for accurate resolution. This causes high computational cost and potential numerical instability.

The neural network structure of FTTNN  and the settings of choosing quadrature points are the same as test problem \ref{h1}. 
In \Cref{alg}, the maximum epoch number for the Adam optimizer is set to $E_{a}=3000$ with a learning rate $\eta_{a}=0.003$, and for the LBFGS optimizer, it is set to $E_{l}=2000$ with a learning rate $\eta_{l}=0.1$.

For $k=3\pi$, it is easy for FTTNN to obtain a good approximation of the exact solution. From Figure \ref{problem42_plot}, it is clear that the estimated solution at $x_3=0.5$ computed by FTTNN (see Figure \ref{problem42_plot}(b)) looks very similar to the exact solution at $x_3=0.5$  (see Figure \ref{problem42_plot}(a)), and the difference between the two solutions is extremely small.

\begin{figure}[!htb]
	\centering
     \subfloat[][The exact solution $u(\mathbf{x})$]{\includegraphics[width=.3\textwidth]{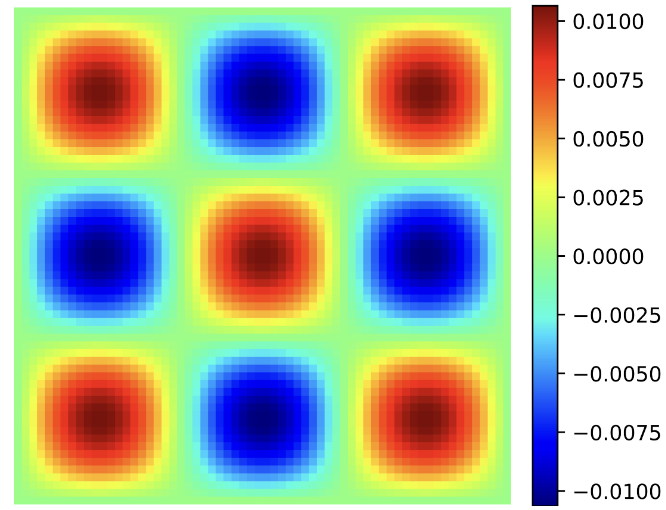}}
     \subfloat[][FTTNN solution $u(\mathbf{x};\theta)$]{\includegraphics[width=.3\textwidth]{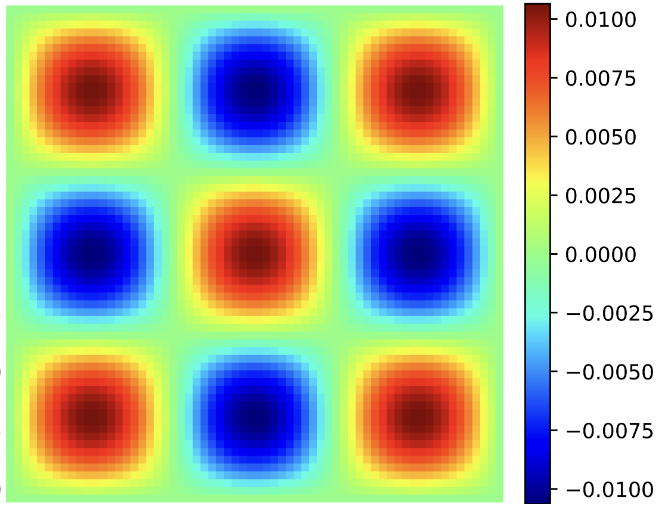}}
     \subfloat[][$u(\mathbf{x};\theta)-u(\mathbf{x})$]{\includegraphics[width=.285\textwidth]{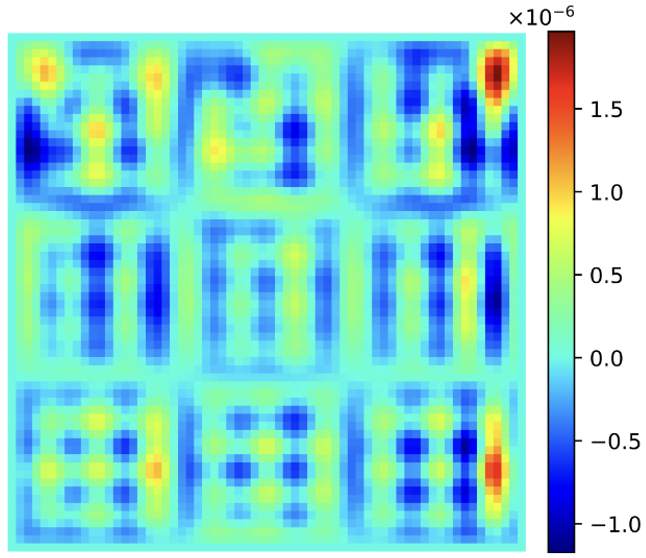}}
	\caption{Solutions of three-dimensional Helmholtz equation with $k=3\pi$ at $x_3=0.5$, test problem \ref{high}.}
    \label{problem42_plot}
\end{figure}
However, for higher wave number, we need a good initialization for each subnetwork of FTTNN. The trained subnetworks of FTTNN for $k=3\pi$ is applied to initialize the subnetworks of FTTNN for $k=5\pi$.
\Cref{h5} corresponds to the relative errors during the training process for $k=5\pi$, where it is seen that the relative errors decrease gradually with the number of epochs increasing and eventually converge to a stable value of $2.6\times 10^{-4}$. \Cref{problem42_plot_5pi}(a) is the exact solution at $x_3=0.5$  for this test problem. \Cref{problem42_plot_5pi}(b) means the estimated solution at $x_3=0.5$ by FTTNN. Their difference is shown in \Cref{problem42_plot_5pi}(c), where it is evident that the pointwise error is extremely small.
\begin{figure}[!htb]
    \centering
    \includegraphics[width=0.7\linewidth]{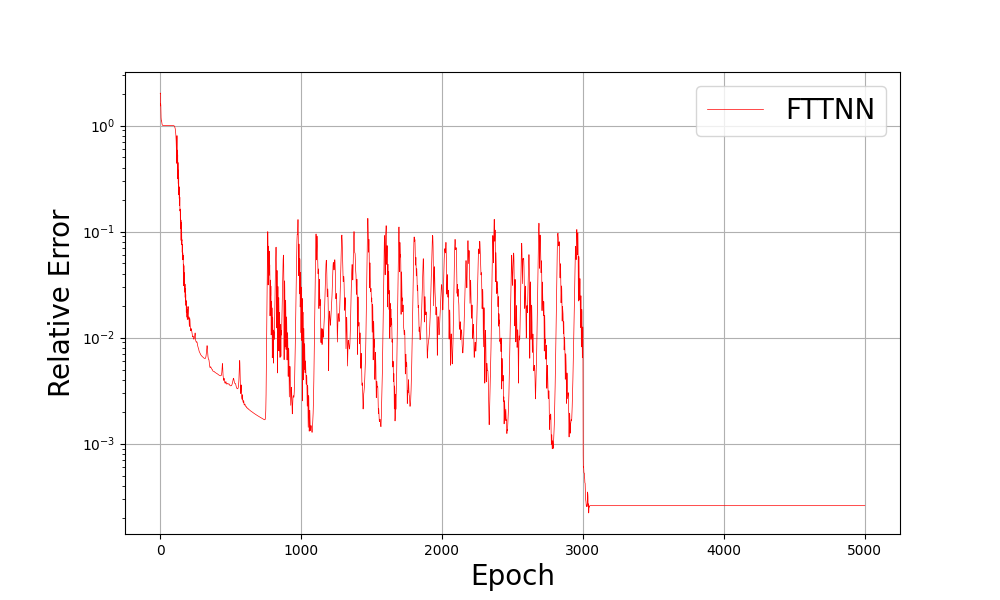}
    \caption{Relative error vs epoch ($k=5\pi, d=3, r=2$), test problem \ref{high}}
   \label{h5}
\end{figure}
\begin{figure}[!htb]
	\centering
     \subfloat[][The exact solution $u(\mathbf{x})$]{\includegraphics[width=.3\textwidth]{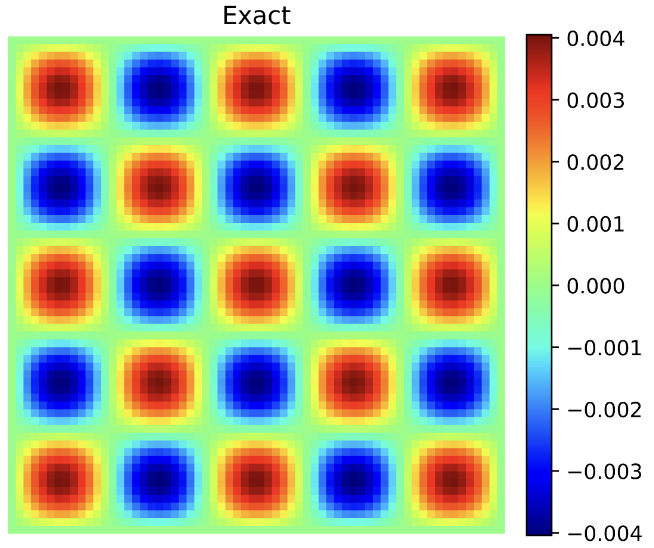}}
     \subfloat[][FTTNN solution $u(\mathbf{x};\theta)$ ]{\includegraphics[width=.3\textwidth]{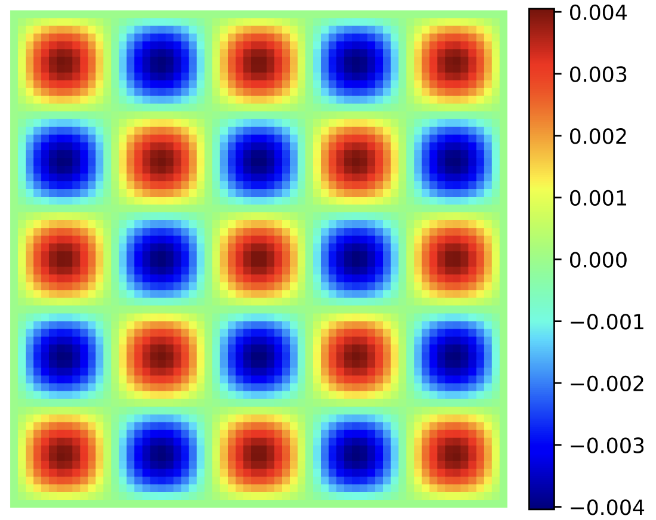}}
     \subfloat[][$u(\mathbf{x};\theta)-u(\mathbf{x})$ ]{\includegraphics[width=.285\textwidth]{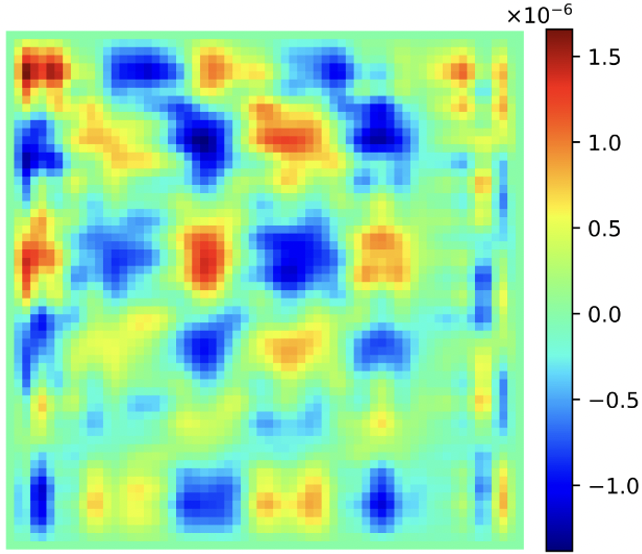}}
	\caption{Solutions of three-dimensional Helmholtz equation with $k=5\pi$ at $x_3=0.5$, test problem \ref{high}.}
    \label{problem42_plot_5pi}
\end{figure}

Furthermore, we apply the trained subnetworks of FTTNN for $k=5\pi$ to initialize the subnetworks of FTTNN for $k=10\pi$. We repeat this step to obtain the solutions of the Helmholtz equation with $k=15\pi, 20\pi, 25 \pi$.
\Cref{test44} provides the initialization of the subnetworks of FTTNN for different wave numbers and the associated relative errors. \Cref{problem42_plot_15pi}(b) provides the estimated solution at $x_3=0.5$ computed by FTTNN for $k=15\pi$, which aligns with the exact solution at $x_3=0.5$ shown in \Cref{problem42_plot_15pi}(a). \Cref{problem42_plot_15pi}(c) is the difference between the FTTNN solution and the exact solution, which implies that the pointwise error remains small.

\begin{table}[htbp]
 \caption{Relative error for different wave numbers, test problem \ref{high}.}\label{test44}
\begin{center}
  \begin{tabular}{|c|c|c|} \hline
   Wave number $k$ &  Initialization &Relative error\\ \hline
 $3\pi$ & Default Initialization & $7.0\times 10^{-5}$\\
			$5\pi$ & FTTNN with $3\pi$ & $2.6\times 10^{-4}$\\
            $10\pi$ & FTTNN with $5\pi$ &  $9.0\times 10^{-4}$\\
            $15\pi$ & FTTNN with $10\pi$ & $2.6\times 10^{-3}$\\
             $20\pi$ & FTTNN with $15\pi$ & $1.9\times 10^{-3}$\\
              $25\pi$ & FTTNN with $15\pi$ &  $2.0\times 10^{-3}$\\ \hline
  \end{tabular}
\end{center}
\end{table}

\begin{figure}[!htb]
	\centering
     \subfloat[][The exact solution $u(\mathbf{x})$]{\includegraphics[width=.3\textwidth]{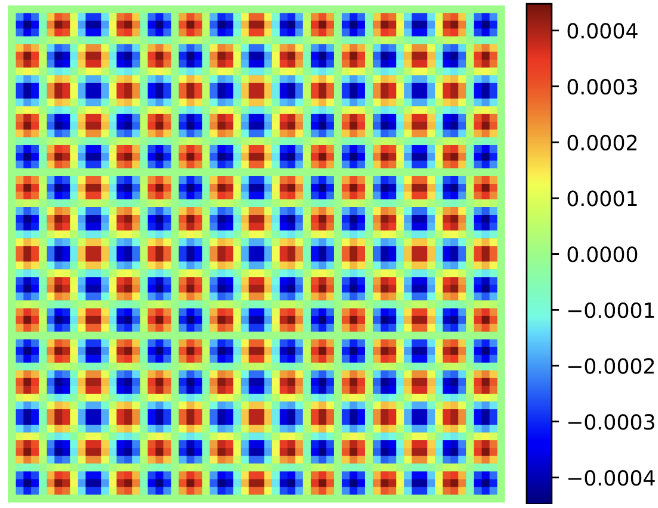}}
     \subfloat[][FTTNN solution $u(\mathbf{x};\theta)$ ]{\includegraphics[width=.3\textwidth]{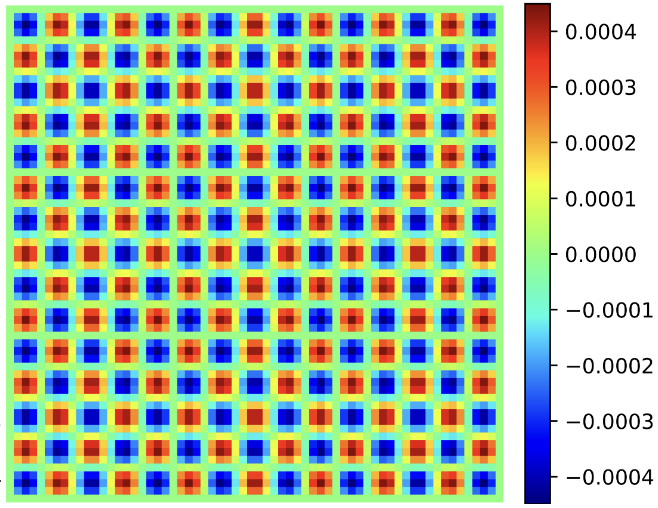}}
     \subfloat[][$u(\mathbf{x};\theta)-u(\mathbf{x})$ ]{\includegraphics[width=.285\textwidth]{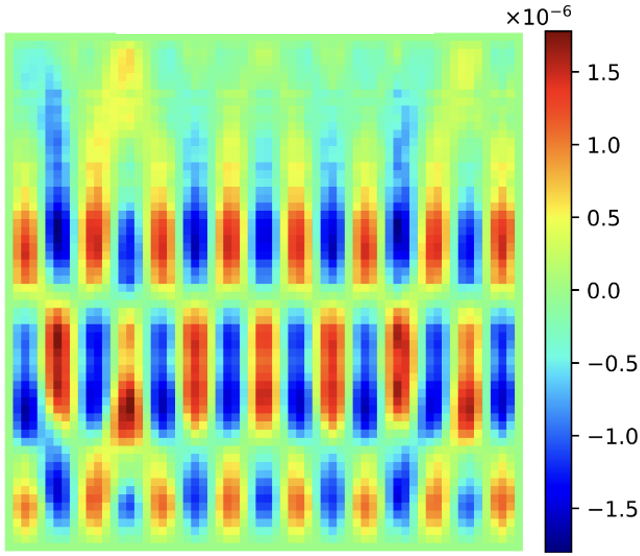}}
	\caption{Solutions of three-dimensional Helmholtz equation with $k=15\pi$ at $x_3=0.5$, test problem \ref{high}.}
    \label{problem42_plot_15pi}
\end{figure}

\subsection{Test problem 5: High-dimensional Schrödinger Eigenvalue Problems}
The Dirichlet eigenvalue problem for the Schrödinger operator is considered,
\begin{equation*}
    \mathcal{H}u(\mathbf{x}):=-\Delta u(\mathbf{x})+V(\mathbf{x})u(\mathbf{x})=\lambda u(\mathbf{x}),\quad \mathbf{x} \in \Omega=(-1,1)^d,\\
\end{equation*}
with homogeneous Dirichlet boundary conditions. Here, $\mathcal{H}$ is the Schrödinger operator and the potential function is defined as $
    V(x_1,\dots,x_d)=\frac{1}{d}\sum_{i=1}^d\cos(\pi x_i+\pi)
$.
Our goal is to compute the first eigenvalue $\lambda_1$,
\begin{equation}\label{eigen_loss}
\lambda_1=\min_{u(\mathbf{x})}\frac{\int_{\Omega} \nabla u(\mathbf{x}) \cdot \nabla u(\mathbf{x}) d\mathbf{x}+ \int_{\Omega} V(\mathbf{x})u^2(\mathbf{x}) d\mathbf{x}}{\int_{\Omega} u^2(\mathbf{x}) d\mathbf{x}}.
\end{equation}

Two cases ($d=5, 10$) are considered. 
FTTNN is compared with DRM, which has one hidden layer with 100 neurons.
 The hidden layer size and the ranks of FTTNN are specified in \Cref{setting4} to ensure that both models have the same total number of parameters.
For computing high-dimensional integrations of \Cref{eigen_loss} in FTTNN,
each $\Omega_i$ is decomposed into 20 equal 
subintervals and choose 20 quadrature points on each subinterval.
For using sampling methods to compute \Cref{eigen_loss} in DRM, the number of sample points for $d=5,10$ is set to 20,000, 50,000 respectively.
The boundary conditions for both FTTNN and DRM are enforced in accordance with test problem 3.
In \Cref{alg}, we set the maximum epoch number for the Adam optimizer to $E_{a}=3000$ with a learning rate $\eta_{a}=0.0001$, and for the LBFGS optimizer, the maximum epoch number is $E_{l}=4000$ with a learning rate $\eta_{l}=0.0001$.
DRM is trained using the same procedure.

\begin{table}[htbp]
 \caption{The settings for FTTNN, test problem 5.}
    \label{setting4}
\begin{center}
  \begin{tabular}{|c|c|c|} \hline
     $d$ &   $r$ & hidden layer size   \\ \hline
5 & 3  &50\\
        10 & 5 &25\\ \hline
  \end{tabular}
\end{center}
\end{table}

\Cref{results4} provides the exact first eigenvalue and the estimated first eigenvalue by FTTNN and DRM in the case of $d=5,10$, where FTTNN produces more accurate first eigenvalue than DRM.
\begin{table}[htbp]
 \caption{The first eigenvalue, test problem 5.}
		\label{results4}
\begin{center}
  \begin{tabular}{|c|c|c|c|} \hline
     $d$ &   Exact &FTTNN &DRM \\ \hline
5 & 11.8345  &11.8345&11.8645\\
            10 &  24.1728&24.1728&23.2659\\ \hline
  \end{tabular}
\end{center}
\end{table}

\section{Conclusion}\label{section5}
In this paper, we propose FTTNN for solving high-dimensional PDEs.
We formulate the PDE solution as a functional tensor train representation, with TT-core functions parameterized by FCNs. Meanwhile, we develop fundamental concepts including functional tensor train rank, and provide the approximation property of FTTNN.
Numerical experiments demonstrate that, compared with discrete tensor train decomposition, FTTNN exhibits greater flexibility in function approximation and is more suitable for solving PDEs on irregular domains. Furthermore, the results indicate that, in comparison with PINN and DRM, FTTNN attains superior accuracy.
In the future work, FTTNN can be applied to more complex high-dimensional problems, such as Fokker-Planck equations.

\section*{Acknowledgments}
The simulations are performed using research computing facilities provided by Information Technology Services, The University of Hong Kong.

\bibliographystyle{siamplain}
\bibliography{references}
\end{document}